\documentclass[11pt]{article}

\usepackage{amssymb}
\usepackage{amsthm}
\usepackage{amsmath}

\newtheorem{theorem}{Theorem}
\newtheorem{lemma}[theorem]{Lemma}

\newtheorem{corollary}[theorem]{Corollary}

\theoremstyle{definition}
\newtheorem{definition}{Definition}

\newtheorem*{remark*}{Remark}

\newcommand\dto{\overset{\mathrm{d}}{\to}}
\newcommand\RR{{\mathbb R}}
\newcommand\cD{\mathcal{D}}
\newcommand\cF{\mathcal{F}}
\newcommand\cB{\mathcal{B}}
\newcommand\cI{\mathcal{I}}
\newcommand\cL{\mathcal{L}}

\newcommand\bb[1]{\bigl(#1\bigr)}
\renewcommand{\ge}{\geqslant}
\renewcommand{\geq}{\geqslant}
\renewcommand{\le}{\leqslant}
\renewcommand{\leq}{\leqslant}
\renewcommand{\Pr}{{\mathbb P}}
\newcommand\E{\operatorname{\mathbb E{}}}
\newcommand{\cc}{\mathrm{c}}
\newcommand\Var{\operatorname{Var}}
\newcommand\eps{\varepsilon}
\newcommand\Bi{\mathrm{Bin}}
\newcommand\floor[1]{\lfloor #1\rfloor}
\newcommand\bigabs[1]{\bigl|#1\bigr|}

\newcommand\Roesler{R\"{o}sler}

\newcommand\txi{\tilde{\xi}}

\newcommand\ceil[1]{\lceil #1\rceil}
\newcommand{\dx}{\mathrm{d}x}
\newcommand{\dda}{\frac{\mathrm{d}}{\mathrm{d}\alpha}}

\begin{document}
\title{A local limit theorem for {\tt Quicksort} key comparisons via multi-round smoothing}
\author{B\'ela Bollob\'as%
\thanks{Department of Pure Mathematics and Mathematical Statistics,
Wilberforce Road, Cambridge CB3 0WB, UK and
Department of Mathematical Sciences, University of Memphis, Memphis TN 38152, USA.
E-mail: {\tt b.bollobas@dpmms.cam.ac.uk}.}
\thanks{Research supported in part by NSF grant DMS-1301614 and EU MULTIPLEX grant 317532.}
\and James Allen Fill%
\thanks{Department of Applied Mathematics and Statistics,
The Johns Hopkins University,
Baltimore, MD 21218-2682, USA.
E-mail: {\tt jimfill@jhu.edu}.}
\thanks{Research supported by the Acheson J. Duncan Fund for the Advancement of Research
in Statistics.}
\and Oliver Riordan%
\thanks{Mathematical Institute, University of Oxford, Radcliffe Observatory Quarter, Woodstock Road, Oxford OX2 6GG, UK.
E-mail: {\tt riordan@maths.ox.ac.uk}.}}
\date{January~16, 2017}
\maketitle

\begin{abstract}
As proved by R\'{e}gnier~\cite{Reg} and R\"osler~\cite{Roesler}, the number of key comparisons required by the randomized sorting algorithm {\tt QuickSort} to sort a list of~$n$ distinct items (keys) satisfies a global distributional limit theorem. Fill and Janson~\cite{FJ2,FJ1} proved results about the limiting distribution and the rate of convergence, and used these to prove a result part way towards a corresponding local limit theorem. In this paper we use a multi-round smoothing technique to prove the full local limit theorem.
\end{abstract}

\section{Introduction}

{\tt QuickSort}, a basic sorting algorithm, may be described as follows.
The input is a list, of length $n\ge 0$, of distinct real numbers (say).
If $n=0$ or $n=1$, do nothing (the list is already sorted). Otherwise, pick an element of the list
uniformly at random to use as the \emph{pivot}, and compare each other element with
the pivot. Recursively sort the two resulting sublists, and combine them in the obvious
way, with the pivot in the middle. (Equivalently, one can sort the initial list
randomly, and always use the first element in each (sub)list as the pivot.)
The recursive calls to the algorithm lead
to a tree, the \emph{execution tree}, with one node for each call. Each
node either has no children (if the corresponding list had length $0$ or $1$)
or two children. The main quantity we study here is the random variable $Q_n$, the total number
of comparisons used in sorting a list of $n$ distinct items.

R\'{e}gnier~\cite{Reg} and \Roesler~\cite{Roesler} each established, using
different methods, a distributional limit theorem for $Q_n$,
proving that $(Q_n - \E Q_n) / n \dto Q$ as $n \to \infty$,
where~$Q$ has a certain distribution that can be characterized in a variety of ways---to name one, as the unique fixed point of a certain distributional identity.  Using that distributional identity, Fill and Janson~\cite{FJ2} showed (among stronger results) that the distribution of~$Q$ has a continuous and strictly positive density $f$ on $\RR$.

Fill and Janson~\cite{FJ1} proved bounds on the rate of convergence in various metrics, including the Kolmogorov--Smirnov distance (i.e., sup-norm distance for distribution functions).  Using this and their results about~$f$ from~\cite{FJ2}, they proved a `semi-local' limit theorem for $Q_n$;
see their Theorem~6.1, which is reproduced in large part as Theorem~\ref{thstart} below.
They posed the question \cite[Open Problem~6.2]{FJ1} of whether the corresponding local limit theorem (LLT) holds.
Here we show that the answer is yes, using a multi-round smoothing technique developed
in an initial draft of~\cite{BR}, but not used in the final version of that paper.
This method may well be applicable to other distributions in which one can find `smooth parts'
on various different scales, including other distributions obeying recurrences of a type
similar to that obeyed by $Q_n$. Taking the `semi-local' limit theorem of~\cite{FJ1} as a starting point, in this paper we shall prove the following LLT for $Q_n$, together with an explicit (but almost certainly not sharp) rate of convergence.

\begin{theorem}\label{th1}
Defining $Q_n$ and $Q$ as above, and setting $q_n := \E Q_n$,
there exists a constant $\eps>0$ such that 
the following holds.
We have
\begin{equation}\label{th1eq}
 \Pr(Q_n = x) = n^{-1} f((x - q_n) / n) + O(n^{-1-\eps})
\end{equation}
uniformly in integers~$x$,
where $f$ is the continuous probability density function of~$Q$.
\end{theorem}

In fact, our proof of Theorem~\ref{th1} gives a bound of the form $O(n^{-19/18}\log n)$ on the error probability in~\eqref{th1eq}.

The basic idea used in our proof, that of strengthening a distributional (often normal) limit theorem
to a local one by smoothing, is by now quite old.
Suppose that $X_n$ takes integer values, and that we know that
\begin{equation}\label{DLT1}
 (X_n-\mu_n)/\sigma_n \dto X,
\end{equation}
for some
nice distribution $X$ (say with continuous, strictly positive density~$f$ on $\RR$).
By the corresponding LLT we mean the statement that whenever $x_n$ is a sequence of
deterministic values with $x_n=\mu_n+O(\sigma_n)$ then
\begin{equation}\label{LLT1}
 \Pr(X_n=x_n) = \sigma_n^{-1} f( (x_n-\mu_n)/\sigma_n) + o(\sigma_n^{-1}).
\end{equation}
It is not hard to see that to deduce \eqref{LLT1} from \eqref{DLT1}, it suffices
to show that `nearby' values have similar probabilities, i.e., that if
$x_n,x_n'=\mu_n+O(\sigma_n)$ and $x_n-x_n'=o(\sigma_n)$, then 
\begin{equation}\label{diff1}
 \Pr(X_n=x_n)=\Pr(X_n=x_n')+o(\sigma_n^{-1}).
\end{equation}

In turn, to prove \eqref{diff1} we might (as in MacDonald~\cite{McDonald})
try to find a `smooth part'
within the distribution of $X_n$. More precisely, we might try to write $X_n=A_n+B_n$
where, for some $\sigma$-algebra $\cF_n$, we have that $A_n$ is $\cF_n$-measurable
and the conditional distribution of $B_n$ given $\cF_n$ obeys (or nearly always obeys)
a relation corresponding to \eqref{diff1}. Then it follows easily (by first considering conditional
probabilities given $\cF_n$) that \eqref{diff1} holds. One idea is
to choose $\cF_n$ so that $B_n$ has a very well understood distribution, such
as a binomial one.

In some contexts, this approach works directly. Here (as far as we can see) it does not.
We can decompose $Q_n$ as above with $B_n$ binomial (see Lemma~\ref{expandbin}), but $B_n$ will have variance $\Theta(n)$,
whereas $\Var Q_n=\Theta(n^2)$. This would, roughly speaking, allow us to
establish that $\Pr(Q_n=x_n)$ and $\Pr(Q_n=x_n')$ are similar for $x_n-x_n'=o(\sqrt{n})$,
but we need this relation for all $x_n-x_n'=o(n)$.\footnote{Actually, since \cite{FJ1}
already contains a `semi-LLT', it  would suffice to consider $x_n-x_n'=O(n^{5/6})$.}

The key idea, as in the draft of~\cite{BR},
is not to try to jump straight from the global
limit theorem to the local one, but to proceed in stages.\footnote{A related idea has recently been used (independently) by Diaconis and Hough~\cite{DH}, in a different context. They work with characteristic functions, rather than directly with probabilities as we do here, establishing smoothness at a range of frequency scales.}
For certain pairs of values
$\ell<m$ with $\ell\ge 1$ and $m=o(n)$ we attempt to show that for any two length-$\ell$
subintervals $I_1$, $I_2$ of an interval $J$ of length $m$ we have
\begin{equation}\label{IJ}
 \Pr(Q_n\in I_1) = \Pr(Q_n\in I_2) + o(\ell/n).
\end{equation}
The distributional limit theorem gives us that for some $m=o(n)$
each interval $J$ of length $m$ has about the right probability, and we then
use the relation above to transfer this to shorter and shorter scales, eventually
ending with $\ell=1$. In establishing \eqref{IJ}, the idea is as before to find
a suitable decomposition $Q_n=A_n+B_n$, but we can use a different decomposition for each
scale---there is no requirement that these decompositions be `compatible' in any
way. For each pair $(\ell,m)$ we need such a decomposition where the distribution of $B_n$
has a property analogous to \eqref{IJ}.

There are some complications carrying this out. Our random variables $B_n$ will have
smaller variances than the original random variables $Q_n$. This means that the 
point probabilities $\Pr(B_n=x_n)$, and (as it turns out) their differences
$\Pr(B_n=x_n)-\Pr(B_n=x_n')$, are too large compared with the bounds we are aiming
for, and the same holds with the points $x_n$ and $x_n'$ replaced by intervals.
For this reason we mostly work with ratios, showing under suitable conditions
that $\Pr(B_n\in I_1)\sim \Pr(B_n\in I_2)$. But this is not always true: Even if $I_1$ and $I_2$
are close, if both are far into a tail of $B_n$ the ratio of the probabilities may be far from $1$.
To deal with this we use another trick: If for some interval $I_1$ there is a significant
probability $p$ that $A_n+B_n\in I$ with the translated interval $I-A_n$ being far above
the mean of $B_n$, say, then there is another interval $J$ (to the left of $I$)
such that there is a probability much larger than $p$ that $A_n+B_n\in J$.
Hence what we will actually show, for a series of scales $m$, is that
(i) each interval of length $m$
has about the right $Q_n$-probability, and
(ii) no interval of length $m$
has $Q_n$-probability much larger than it should.
We will use (ii) at the longer scale $m$ to show that the `tail contributions'
described above are small at scale $\ell$. Thus we will be able to transfer the
combined statement (i)+(ii) from longer to shorter scales.

In the particular context of {\tt QuickSort} there is a very nice way to find binomial-like
smooth parts: we partially expand the execution tree, looking, roughly speaking,
for a way of writing the original instance as the union of $\Theta(s)$ instances
of {\tt QuickSort} each run on $\Theta(r)$ input values, where $s=n/r$. Conditioning
on this partial expansion (plus a little further information) the unknown part
of the distribution is then `binomial-like': it is a sum of $\Theta(s)$ independent
random variables each with `scale' $\Theta(r)$.

The rest of the paper is organized as follows: in Section~\ref{sec_prelim} we state
two standard results we shall need later, and then establish
the existence of the decompositions described in the previous paragraph.
In Section~\ref{sec_binom} we prove some simple properties of `binomial-like'
distributions. Section~\ref{sec_core} is the heart of the paper; here we present
the core smoothing argument, showing how to transfer `smoothness' from a scale
$m$ to a scale $\ell\le m$ under suitable conditions. In Section~\ref{sec_complete}
we complete the proof of Theorem~\ref{th1}; this is a matter of applying the results
from Section~\ref{sec_core} with suitable parameters, taking
as a starting point the `semi-local' limit theorem established by Fill and Janson~\cite{FJ1}.
Finally, in Section~\ref{sec_soft}
we outline a different way of applying the same smoothing results, taking a weaker
distributional convergence result as the starting point; this may be applicable in other settings.

\section{Preliminaries}\label{sec_prelim}

\subsection{Some standard inequalities}
We shall use the Azuma--Hoeffding inequality (see \cite{Azuma} and \cite{Hoeff})
in the following form (see, for example, Ross~\cite[Theorem 6.3.3]{Ross}).
\begin{theorem}\label{Azuma}
Let $(Z_n)_{n \geq 1}$ be a martingale with mean $\mu = \E Z_n$. Let $Z_0 = \mu$ and suppose that for
nonnegative constants $\alpha_i, \beta_i$, $i \geq 1$, we have
\[
- \alpha_i \leq Z_i - Z_{i - 1} \leq \beta_i.
\]
Then, for any $n \geq 0$ and $a \geq 0$ we have
\[
\Pr(Z_n - \mu\geq a) \leq \exp\left\{ - 2 a^2 \left/ \sum_{i = 1}^n (\alpha_i + \beta_i)^2 \right. \right\},
\]
and the same bound applies to $\Pr(Z_n - \mu \leq - a)$. \qed
\end{theorem}

We shall also need Esseen's inequality, also known as the Berry--Esseen Theorem; see, for example, Petrov~\cite[Ch. V, Theorem 3]{Petrov}.
We write~$\Phi$ for the distribution
function of the standard normal random variable.
\begin{theorem}\label{BE}
Let $Z_1,\ldots,Z_t$ be independent random variables with $\rho=\sum_{i=1}^t \E(|Z_i|^3)$ finite,
and let $S=\sum_{i=1}^t Z_i$.
Then
\[
 \sup_x \bigabs{ \Pr( S\le x ) - \Phi((x-\mu)/\sigma) } \le A\rho/\sigma^3,
\]
where $\mu$ and $\sigma^2$ are the mean and variance of $S$,
and $A$ is an absolute constant.\qed
\end{theorem}
 
\subsection{Decomposing the execution tree}

In this subsection we shall show that, given a parameter $r$,
a single run of {\tt QuickSort} on a list of length $n$ will, with high probability, involve
$\Omega(n/r)$ instances of {\tt QuickSort} run on disjoint lists of length
between $r/2$ and $r$.

Let $2\le r< n$ be integers. We can implement {\tt QuickSort} on a list of length $n$ in
two phases as follows: in the first step of Phase I, pick the random pivot dividing
the original list into two sublists of total length $n-1$.  In step $t$ of Phase I, if
all the current sublists have length at most $r$, do nothing.  Otherwise, pick a sublist
of length at least $r+1$ arbitrarily, and pick the random pivot in this sublist, dividing
its remaining elements into two new sublists. After $n$ steps, we proceed to Phase II,
where we simply run {\tt QuickSort} on all remaining sublists.
Let $X_{n,r}$ denote the number of sublists at the end of Phase I that have length
between $r/2$ and $r$.

\begin{lemma}\label{expandA}
Let $r \geq 20$ be even and $n\ge 5r$.
Then
\[
 \Pr\left(X_{n,r}\le \tfrac{n}{3r}\right) \le e^{-n/(400r)}.
\]
\end{lemma}

\begin{proof}
We have specified that~$r$ be even only for convenience.
We have made no attempt to optimize the values of the various constants; these will be
irrelevant later.

Running {\tt QuickSort} in two phases as above, let $T$ be the number of
`active' steps in Phase I, i.e., steps in which we divide a sublist into two.
Clearly, $T\le n$, the first $T$ steps of Phase I are active, and after $T$
steps we have $T+1$ sublists of total length $n-T$. The idea of the proof
is to show that $T$ is very unlikely to be larger than $20n/r$, say,
that $\E X_{n,r}$ is of order $n/r$, and that each decision in the first
phase of our algorithm alters the conditional expectation of $X_{n,r}$
by at most $1$. The result will then follow from the Azuma--Hoeffding inequality.

Throughout the proof we keep $r\ge 20$ fixed.
Let
\[
 t_0=\ceil{20n/r}.
\]
Observe that if $T\ge t_0$, then after step $t_0$
we have $t_0+1$ sublists with total length $<n$. Since at most $10n/r\le t_0/2$
of these sublists can have length at least $r/10$, at least $t_0/2$ of our sublists
have length $<r/10$.
Let $N$ be the number of sublists after $t_0$ steps that have length
less than $r/10$, so we have shown that
\[
 \Pr( T \ge t_0) \le \Pr (N\ge t_0/2).
\]
In any step of Phase I, we either do nothing,
or randomly divide a list of some length $\ell\ge r+1$. In the latter
case, the (conditional, given the past) probability of producing a sublist
of length $<r/10$ is at most
\[
 2\frac{(r/10 +1)}{\ell} \le \frac{ 3r/10}{\ell} < \frac{3}{10},
\]
since $r\ge 20$ and $\ell>r$. It follows that $N$ is stochastically dominated by a binomial
distribution with parameters $t_0$ and $3/10$, so
\begin{equation}\label{Tbd}
 \Pr\bb{T\ge t_0} \le \Pr\bb{N\ge t_0/2} \le \Pr\bb{ \Bi(t_0,3/10) \ge t_0/2 } \le e^{-2t_0/25},
\end{equation}
using Theorem~\ref{Azuma}, or a standard Chernoff bound, for the last step.
 
Turning to the next part of the argument, as $r$ is fixed throughout,
let us write $X_n$ for the random variable $X_{n,r}$.
We extend the definition of $X_n$ to the case $n\le r$ by considering
Phase I to end immediately (with one `sublist' of length $n$) in this case.
The sequence $(X_n)$ satisfies the deterministic initial conditions
\begin{align*}
&X_0 = \cdots = X_{(r / 2) - 1} = 0, \\
&X_{r / 2} = \cdots = X_r = 1,
\end{align*}
and (considering the first step in Phase I as described above)
the distributional recurrence relation
\begin{equation}\label{Xdistrecur}
X_n \stackrel{\cal L}{=} X_{U_n - 1} + X^*_{n - U_n}, \quad n \geq r + 1,
\end{equation}
where, on the right, $X_j$ and $X^*_j$ are independent probabilistic copies of $X_j$ for each $j = 1, \dots, n - 1$ and $U_n$ is uniformly distributed on $\{1, \dots, n\}$, and is independent of all the $X$ and $X^*$ variables.
Let
\[
\xi_n := \E X_n.
\]
From \eqref{Xdistrecur} we have $\xi_n=\frac{2}{n}\sum_{i=0}^{n-1}\xi_i$ for $n\ge r+1$.
It follows that
\begin{align*}
&\xi_0 = \cdots = \xi_{(r / 2) - 1} = 0, \\
&\xi_{r / 2} = \cdots = \xi_r = 1,
\end{align*}
and
\begin{equation}\label{xi}
\xi_n = \frac{n + 1}{r + 1}, \quad n \geq r + 1.
\end{equation}
(The last equation holds also for $n = r$.)
Define $\txi_n=\frac{n+1}{r+1}$ for all $n$.
Then $\txi_{k-1}+\txi_{n-k}=\txi_n$
always. Since
\[
 |\xi_n-\txi_n| \le \frac{r/2}{r+1} < \frac{1}{2},
\]
it follows that if $n\ge r+1$ (and so $\xi_n=\txi_n$), then
\begin{equation}\label{MLip}
 -1 <  \xi_{k - 1} + \xi_{n - k} - \xi_n  < 1
\end{equation}
for all $1 \leq k \leq n$. 

Let $\cF_t$ denote the $\sigma$-algebra corresponding to information revealed in
the first $t$ steps of Phase I as described above.
Define
\[
 M_t = \E[X_n\mid \cF_t],
\]
so that $(M_t)_{t=0}^n$ is a (Doob) martingale.
It follows from \eqref{MLip} that the martingale $(M_t)$, which has mean $M_0=\xi_n$ given 
by~\eqref{xi}, satisfies
\[
 -1 < M_t - M_{t - 1} < 1
\]
for every~$t$.

Let $E$ be the event that $X_n\le \frac{n}{3r}$.
Since 
\[
 \xi_n = \frac{n+1}{r+1} \ge \frac{n}{r+1} \ge \frac{2n}{3r},
\]
when $E$ holds we have $X_n-\xi_n\le -\frac{n}{3r}$.
After the first $T$ steps of Phase I, nothing further happens, so $M_T=M_{T+1}=\cdots=M_n=X_n$.
Hence, writing $t_0=\ceil{20n/r}$ as before, we have
\[
 \Pr(E) \le \Pr\left(T>t_0\right) + \Pr\left(M_{t_0}-\xi_n \le -\tfrac{n}{3r}\right).
\]
By \eqref{Tbd} and the Azuma--Hoeffding
inequality (Theorem~\ref{Azuma}), it follows that
\begin{eqnarray*}
 \Pr(E) 
 &\le& e^{-2t_0/25} +  \exp \left( - \frac{n^2}{18r^2t_0} \right) \\ 
&\le& \exp \left( - \frac{40n}{25r} \right) + \exp \left( - \frac{n}{378 r} \right)  \\
&\le& \exp \left( -   \frac{n}{400 r}  \right),
\end{eqnarray*}
where the penultimate inequality holds 
because $20n/r \le t_0 \le 21 n/r$, since
$n/r\ge 5$, and the final
inequality holds because $e^{- 8x / 5} + e^{- x / 378} \leq e^{- x / 400}$ for $x \geq 5$.
\end{proof}

\begin{corollary}\label{Cexpand}
Let $r \geq 20$ be even and  $n \ge 5 r$.
Then we may write $Q_n=A+B$ where, for some $\sigma$-algebra $\cF$, we have that $A$
is $\cF$-measurable, and, with probability at least
$1 - e^{-n/(400r)}$,
the conditional
distribution of $B$ given $\cF$ is the sum of
$s = \lceil n/(3r) \rceil$ independent random variables $B_1, \ldots, B_s$
with each $B_i$ having the distribution $Q_{r_i}$ for some $r_i$ with $r / 2 \le r_i \le r$.
\end{corollary}

\begin{proof}
Run {\tt QuickSort} in two phases as above, and define $X_{n,r}$
as in Lemma~\ref{expandA}. Let $E$ be the event that $X_{n,r}\ge s = \lceil n/(3r) \rceil$,
so $\Pr(E)\ge 1 - e^{-n/(400r)}$ by Lemma~\ref{expandA}.
We now subdivide Phase II into two parts.
When $E$ holds, we select
$s$ sublists from the end of Phase I with length between $r/2$ and $r$, otherwise
we do not select any. In Phase IIa, we run {\tt QuickSort} on all sublists
\emph{except} the selected ones. In Phase IIb, we run {\tt QuickSort} on the
selected sublists.
Take the $\sigma$-algebra $\cF$ to be the $\sigma$-algebra corresponding to all the information uncovered in Phases I and IIa, and $A$ to be the total number
of comparisons made during Phases I and IIa. Take $B_i$, $i = 1, \dots, s$, to be (when~$E$ occurs) the number of comparisons involved in running {\tt QuickSort} on the $i$th selected
sublist.
\end{proof}

\subsection{Truncating the summands}

The sum of the $B_i$ above will roughly serve as our `binomial-like' distribution, but
we would like a little more information about it. Knowing that $B_i$ has `scale' roughly
$r_i\approx r$, we shall condition on $|B_i-\E B_i|$ being at most $2r_i$.
This will still keep a constant fraction of the variance, while giving us better control
on the distribution of the sum of such random variables.

Writing $q_n$ for $\E Q_n$, for $n \geq 1$ let $Q_n^*=(Q_n-q_n)/n$ denote the centered and normalized form of $Q_n$.
Since $Q_n^*$ converges in distribution to $Q$, a distribution with a continuous positive density on $\RR$,
we know that there are constants $n_0$ and $c_1>0$
such that for all $n\ge n_0$ we have, say, $\Pr(Q_n^*\in [-2,-1])\ge c_1$ and 
$\Pr(Q_n^*\in [1,2])\ge c_1$. Hence, for $n\ge n_0$,
\begin{equation}\label{Qn22}
 \Pr(Q_n^*\in [-2,2]) \ge 2c_1
\end{equation}
and, since $\Pr(Q_n^*\in I \mid Q_n^*\in [-2,2])\ge c_1/1=c_1$ for $I = [-2, 1]$ and $I = [1, 2]$, we have
\[
 \Var(Q_n^* \mid Q_n^* \in [-2,2]) \ge c_1.
\]
Let $W_n'$ denote the distribution of $Q_n^*$ conditioned to lie in $[-2,2]$,
and let $W_n := W_n' - \E W_n'$. Then $|W_n| \le 4$ and $\Var W_n\ge c_1$. We will record
the consequences for the unrescaled distribution of $Q_n$ immediately after the following definition.

\begin{definition}
\label{Drdef}
Given $r>0$ let $\cD_r$ denote the set of probability distributions
of random variables $X$ with the following properties: $\E X=0$, $|X|\le 4r$,
and $\Var X\ge c_1 (r/2)^2$.
\end{definition}

The calculations above have the following simple consequence: for any $r \ge 2 n_0$
and any $r'$ satisfying $r/2\le r'\le r$,
we have $\Pr(Q_{r'}\in [q_{r'}-2r',q_{r'}+2r'])\ge 2c_1$, and the conditional distribution
of $Q_{r'}$ given this event is of the form $z_{r'}+X_{r'}$ for some constant $z_{r'}$
and some $X_{r'}$ with law in $\cD_r$.

\begin{definition}\label{Brsdef}
Given $r>0$ and a positive integer $s$, let $\cB_{r,s}$ denote the set of 
$s$-fold convolutions of distributions from $\cD_r$.
\end{definition}

In other words, $X$ has a distribution in $\cB_{r,s}$ if we can write
$X=X_1+\cdots+X_s$ where the $X_i$ are independent and each has law in $\cD_r$.
The distributions in $\cB_{r, s}$ will be the `binomial-like' ones we shall use in the smoothing argument.

\begin{remark*}
More properly we should write $\cD_{r,c_1}$ and $\cB_{r,s,c_1}$ for the classes
defined in Definitions~\ref{Drdef} and~\ref{Brsdef}.
In this paper we need only consider a particular value of $c_1$ as at the start 
of this subsection, but in other contexts one might consider these classes
for other values of $c_1$. The results below of course extend to this setting.
\end{remark*}

The next lemma, a simple consequence of Corollary~\ref{Cexpand}, will play a key role in our smoothing arguments.

\begin{lemma}\label{expand2}
There are positive constants $r_0$, $c_2$ and $c_3$
such that the following holds whenever $n$ and $r$ are positive integers with $r$ even
and $r_0\le r \le c_2n$:
we may write $Q_n=A+B$ where, for some $\sigma$-algebra $\cF$, we have that $A$
is $\cF$-measurable, and, with probability at least $1-e^{-c_3n/r}$, the conditional
distribution of $B$ given $\cF$ is in the class $\cB_{r,s}$, with $s=\ceil{c_2 n/r}$.
\end{lemma}
\begin{proof}
We start by taking $\cF'$, $A'$, and $B'$ to be as in Corollary~\ref{Cexpand}.
Let $E'\in \cF'$ be the event that we may write the conditional distribution of $B'$
as the sum of independent variables $B_1',\ldots,B_t'$, $t=\lceil n / (3r) \rceil$,
with $B_i'$ having (conditionally given $\cF'$) the distribution of $Q_{r_i}$ for
some $r/2\le r_i\le r$.
By Corollary~\ref{Cexpand} we have $\Pr(E')\ge 1-e^{-\Omega(n/r)}$.
We choose $c_2\le c_1/6$, and set $s = \lceil c_2 n / r \rceil$.
Note that $c_2n/r\ge 1$, so $s\le 2c_2n/r$.
We shall
reveal certain extra information as described in a moment. Let $E_i$ denote the event that
$B_i'\in [q_{r_i}-2r_i,q_{r_i}+2r_i]$,
and let $E$ denote the event that at least 
$s$ of the events $E_i$ occur. Each event $E_i$ has conditional probability at least $2c_1$
by \eqref{Qn22}. Since
the $E_i$'s are conditionally independent given $\cF'$, and $c_1t\ge 2c_2n/r\ge s$,
we see 
[from $\Pr(E\mid E') \geq \Pr(\Bi(t, 2 c_1) \geq c_1 t )$ and Chernoff's inequality] 
that $\Pr(E \mid E') \ge 1-e^{-\Omega(t)} = 1-e^{-\Omega(n/r)}$. Hence $\Pr(E)\ge 1-e^{-\Omega(n/r)}$.

The extra information we reveal is as follows:\ firstly, which $E_i$'s occur,
and hence whether $E$ occurs. When $E$ does occur, we let $\cI$ be the set of the first $s$
indices $i$ such that $E_i$ occurs, otherwise we may take $\cI=\emptyset$, say.
We now reveal the values of all $B_i'$, $i\notin \cI$, and set $B=\sum_{i\in\cI} (B_i'-z_{r_i})$.
Let $\cF\supset \cF'$ denote the $\sigma$-algebra generated by all information revealed so far.
Then $A=Q_n-B = A'+\sum_{i\in\cI} z_{r_i} + \sum_{i\notin \cI} B_i'$ is certainly $\cF$-measurable.
Also, when $E$ occurs, the conditional distribution of $B$ given $\cF$ is in $\cB_{r,s}$,
as required.
\end{proof}

\section{Properties of binomial-like distributions}\label{sec_binom}

In this section we establish some simple properties of distributions in the class $\cB_{r,s}$ without
aiming for tight bounds. The first property is asymptotic normality, which will give `smoothness'
on scales larger than $r$.
Here and in what follows all constants are absolute,
except in that they may depend on the absolute constant $c_1$ in the definition of $\cD_r$.

\begin{lemma}\label{Brs1}
For any random variable $X$ with distribution in $\cB_{r,s}$ we have $\Var X=\Theta(r^2 s)$
and
\[
 \Pr\bb{ X \le \E X + x \sqrt{\Var X}} = \Phi(x) + O(1/\sqrt{s}),
\]
where~$\Phi$ is the standard normal distribution function,
and the implicit constants depend only on the constant $c_1$ in Definition~\ref{Drdef}.
\end{lemma}
\begin{proof}
Dividing $X$, and each of the $s$ summands $X_i$ comprising it, through by $r$, we may
assume without loss of generality that $r=1$. Then apply the Berry--Esseen Theorem
(Theorem~\ref{BE} above),
noting that each $X_i$ is bounded in absolute value by~$4$, and so has bounded third moment,
and that $\Var X$ is $\Theta(s)$
(under our assumption that $r = 1$).
\end{proof}

Next we establish a common tail bound for all distributions in the class $\cB_{r, s}$.
\begin{lemma}\label{Brstail}
There are constants $c > 0$ and $C$ such that, for all $X$ with distribution in $\cB_{r,s}$,
all $t\ge 0$, and all $\ell\ge r$ we have
\[
 \Pr(X\in [t,t+\ell]) \le \frac{C\ell}{r\sqrt{s}} e^{-ct^2/(r^2s)}.
\]
\end{lemma}
\begin{proof}
The Azuma--Hoeffding inequality gives that, uniformly for $Y$ with distribution in $\cB_{r,s}$,
we have
\begin{equation}\label{HAY}
 \Pr(Y\ge t)\le e^{-\Omega(t^2/(r^2s))}.
\end{equation}
Separately, for any interval $I$ of length $\ell\ge r$, we have
\begin{equation}\label{YI}
 \Pr(Y\in I) = O(\ell/(r\sqrt{s})).
\end{equation}
Indeed, writing $Y'$ for a Gaussian with the same mean and variance as $Y$,
by Lemma~\ref{Brs1} we have $\Pr(Y\in I)=\Pr(Y'\in I)+O(1/\sqrt{s})$. As $Y'$ has variance
$\Theta(r^2s)$ we have $\Pr(Y'\in I)=O(\ell/(r\sqrt{s}))$, and the error term
is absorbed into the main term by the lower bound on $\ell$. Somewhat surprisingly, the proof can be completed by
multiplying these two bounds.

Indeed, in proving the claimed result, adjusting the constants if needed, we may assume that
$s$ is even.
Then we may write $X=Y+Z$ where $Y$ and $Z$ are independent and have distributions
in the class $\cB_{r,s/2}$. Let $I=[t,t+\ell]$ with $t\ge 0$. Since $X\ge t$ implies either
$Y\ge t/2$ or $Z\ge t/2$, we may write
\begin{equation}\label{tt2}
 \Pr(X\in I) \le \Pr\bb{ Y+Z\in I,\  Y\ge t/2} + \Pr\bb{Y+Z\in I,\  Z\ge t/2}.
\end{equation}
We bound the first term from above by 
\begin{eqnarray*}
 \Pr(Y\ge t/2)\Pr(Y+Z\in I \mid Y\ge t/2) &\le& \Pr(Y\ge t/2) \sup_y \Pr(Y+Z\in I\mid Y=y) \\
 &=& \Pr(Y\ge t/2) \sup_x \Pr(Z\in [x,x+\ell]).
\end{eqnarray*}
The final quantity is 
$e^{-\Omega(t^2/(r^2s))} \times O(\ell/(r\sqrt{s}))$ 
by \eqref{HAY} and \eqref{YI}. The second term in \eqref{tt2} 
may be bounded in the same way.
\end{proof}

\subsection{A tilting lemma}

In proving Lemma~\ref{Brs1} we applied the Berry--Esseen Theorem to distributions from $\cB_{r, s}$; next we shall apply the same result to exponential tilts of these distributions, to prove the following result. In what follows, $c_1$ is the constant appearing
in Definition~\ref{Drdef}.

\begin{lemma}\label{tiltlemma}
Let $K>0$ be constant. There exists a constant $C'=C'(c_1,K)$ such that
the following holds whenever
\begin{equation}\label{tlassump}
 \lambda\ge 1, \quad m\ge \ell\ge C'r \quad \hbox{and} \quad \frac{\lambda m}{r\sqrt{s}} \le K.
\end{equation}
Let~$X$ be a random variable with distribution in~$\cB_{r, s}$, and let $I_1$ and $I_2$ be
subintervals, each of length~$\ell$, of an interval~$J$ of length~$m$ with
$J\subset [ - K \lambda r\sqrt{s}, K \lambda r\sqrt{s}]$.
Then
\[
 \Pr(X\in I_2) = \Pr(X\in I_1) \left(1+O\left( \frac{r }{ \ell} + \frac{\lambda m}{r\sqrt{s}} \right) \right).
\]
\end{lemma}

The proof of Lemma~\ref{tiltlemma} will be based on standard exponential tilting arguments similar to those used around Lemma 6.4 in~\cite{norm2}.
Let $Y$ be a random variable with bounded support. Then for any $\alpha\in \RR$
we may define the \emph{tilted distribution} $\cL(Y^{(\alpha)})$ by
\begin{equation}\label{tiltdef}
 \Pr( Y^{(\alpha)} \in \dx) = \Pr(Y \in \dx) \frac{e^{\alpha x}}{\gamma},
\end{equation}
where $\gamma = \gamma(\cL(Y), a) = \E e^{\alpha Y}$; here $\cL(Y)$ denotes the law, or distribution, of the random variable~$Y$.  

Before starting the proof of Lemma~\ref{tiltlemma}, we establish some elementary
properties of tilted versions of distributions with law in the set $\cD_1$
defined in Definition~\ref{Drdef}. 

\begin{lemma}\label{tiltprep}
There is a constant $c>0$, depending only on $c_1$, such that for any $Y$ with $\cL(Y)\in \cD_1$
and any $\alpha\in [-1,1]$ we have $\Var Y^{(\alpha)}\ge c$.
Furthermore, 
\begin{equation}\label{dEYa}
 \dda \E Y^{(\alpha)} \ge c
\end{equation}
whenever $|\alpha|\le 1$.
\end{lemma}
\begin{proof}
The first statement is intuitively clear: we take a distribution whose variance is bounded from
below, and `distort it' by a bounded amount, so the variance will still be bounded from below.
We spell out a concrete argument, not aiming for the best possible bound.

Let~$Y$ have distribution in $\cD_1$.  Then, recalling Definition~\ref{Drdef}, for any $b > 0$ we have
\[
 \tfrac{1}{4} c_1 \leq \E Y^2 \leq b^2 \Pr(Y^2 \leq b^2) + 16 \Pr(Y^2 > b^2) \leq b^2 + 16 \Pr(|Y| > b).
\]
Take $b = \frac{1}{4} \sqrt{c_1}$.  Then
\[
 \Pr(|Y| > \tfrac{1}{4} \sqrt{c_1}) \geq \tfrac{1}{16} (\tfrac{1}{4} - \tfrac{1}{16}) c_1 
 = \tfrac{3}{256} c_1 > \tfrac{1}{100} c_1.
\]
Without loss of generality we may thus assume that
\[
 \Pr(Y > \tfrac{1}{4} \sqrt{c_1}) >  \tfrac{1}{200} c_1.
\]
Since $\E Y=0$ and $Y$ is supported on $[-4,4]$, it follows that
\[
 \Pr(Y < 0) > \tfrac{1}{4} \cdot \tfrac{1}{4} \sqrt{c_1} \cdot \tfrac{1}{200} c_1 =\tfrac{1}{3200} c_1^{3/2}.
\]
Since $Y$ is supported on $[-4,4]$, for $|\alpha|\le 1$ we have $\gamma=\E e^{\alpha Y}\le e^4$,
while $e^{\alpha x}$ is at least $e^{-4}$ for all $x$ in the support of $Y$.
Hence,
\[
 \Pr(Y^{(\alpha)} > \tfrac{1}{4} \sqrt{c_1}) \ge \frac{e^{-4}}{\gamma} \Pr(Y> \tfrac{1}{4}\sqrt{c_1}) 
 \ge e^{-8} \tfrac{1}{200} c_1.
\]
Similarly, 
\[
 \Pr(Y^{(\alpha)} < 0) \ge e^{-8} \tfrac{1}{3200} c_1^{3/2}.
\]
The last two bounds clearly imply a lower bound on $\Var Y$ that depends only on $c_1$.

To establish \eqref{dEYa}, note that
\[
 \E Y^{(\alpha)} = \frac{ \E(Ye^{\alpha Y}) }{ \E(e^{\alpha Y}) },
\]
so by the quotient rule,
\[
 \dda \E Y^{(\alpha)} = \frac{ \E(Y^2 e^{\alpha Y}) }{ \E(e^{\alpha Y}) } -
\left( \frac{ \E(Y e^{\alpha Y}) }{ \E(e^{\alpha Y}) } \right)^2
 = \Var Y^{(\alpha)}.
\]
\end{proof}

\begin{proof}[Proof of Lemma~\ref{tiltlemma}.]
Let $X = X_1 + \cdots + X_s$ have distribution in $\cB_{r,s}$.
We aim to bound $\Pr(X \in I_2) / \Pr(X \in I_1)$, where $I_1$ and $I_2$ are intervals of length~$\ell$ both contained in a common interval $J$ of length~$m$.  By rescaling (considering $\ell / r$ and $m / r$ in place of~$\ell$ and~$m$) we may assume without loss of generality that $r = 1$.

A simple calculation shows that if we tilt each $X_i$ by the same parameter $\alpha$,
then the independent sum of $X_1^{(\alpha)},\ldots,X_s^{(\alpha)}$ has the same distribution
as $X^{(\alpha)}$. By Lemma~\ref{tiltprep} we thus have
\begin{equation}\label{dda}
 \dda \E X^{(\alpha)}  = \dda \sum_{i=1}^s \E X_i^{(\alpha)}  \ge cs
\end{equation}
for $\alpha\in [-1,1]$.
Let 
\[ 
 I_j = [t_j,t_j+\ell]
\]
for $j=1,2$. Since $r=1$ and $m\ge C'r=C'$ by \eqref{tlassump}, we have
\[
 \frac{|t_1|}{s} \le \frac{K\lambda\sqrt{s}}{s} = \frac{K\lambda}{\sqrt{s}} \le \frac{K\lambda m}{C'\sqrt{s}}\le \frac{K^2}{C'},
\]
using the third assumption in \eqref{tlassump} in the last step.
Hence, choosing $C'$ large enough, we have $|t_1|\le cs$.
Since $\E X^{(0)}=\E X=0$, it follows from \eqref{dda} that
there is a unique value $a\in [-1,1]$ such that
\[
 \E X^{(a)} = t_1.
\]
Moreover, we have
\begin{equation}\label{abound}
 a = O(t_1/s) = O( \lambda s^{-1/2} ).
\end{equation}

From now on we fix this tilting parameter, writing $X_i'$ for $X_i^{(a)}$,
and $X'$ for the independent sum $X_1' + \cdots + X_s'$.
As noted above, $X'$ has the same distribution as $X^{(a)}$.
In other words, 
\[
 \Pr(X' \in \dx) = \Pr(X \in \dx) \frac{e^{a x}}{\gamma},
\]
where $\gamma = \E e^{a X}$ is independent of $x$.
Since $I_1$ and $I_2$ lie in an interval $J$ of length $m$, it follows easily that
\[
 \frac{\Pr(X'\in I_2)}{\Pr(X'\in I_1)}
 = \frac{\Pr(X\in I_2)}{\Pr(X\in I_1)} e^{O(am)}.
\]
Now by \eqref{abound} and \eqref{tlassump},
\[
  a m =O(\lambda ms^{-1/2}) = O(1),
\]
so the $e^{O(am)}$ term is $1+O(am)$. Hence
\begin{equation}\label{ratios}
  \frac{\Pr(X \in I_2)}{\Pr(X \in I_1)}  =  \frac{\Pr(X' \in I_2)}{\Pr(X' \in I_1)} 
\left(1+ O\left(\frac{\lambda m}{s^{1/2}}\right) \right).
\end{equation}
It remains to bound the ratio $\Pr(X' \in I_2)/\Pr(X' \in I_1)$.

Like the distribution of $X_i$, the distribution of $X_i'$ is supported on the interval $[-4, 4]$.
Hence $\Var X_i'=O(1)$. But by the first part of Lemma~\ref{tiltprep}, $\Var X_i'\ge c$,
so $\Var X_i'=\Theta(1)$. Also, the absolute third moment $\E|X_i'|^3$ is clearly
at most $4^3=O(1)$.
The implicit constants in these estimates depend only on $c_1$ and $K$, not on $i$.
Hence,
\begin{align*}
 \mu &:= \E X' = t_1, \\
 \sigma^2 &:= \Var X' = \sum_{i = 1}^s \Var X_i' = \Theta(s), \\
 \rho &:= \sum_{i =1}^s \E |X_i'|^3 = O(s).
\end{align*}
Let us note for later that since $\lambda\ge 1$ and $\lambda m/\sqrt{s}=O(1)$ we have
\[
 m \le \lambda m = O(\sqrt{s}) = O(\sigma),
\]
and hence
\begin{equation}\label{ms2}
 \frac{m^2}{\sigma^2} = O\left(\frac{m}{\sigma}\right) = O\left(\frac{m}{s^{1/2}}\right)
 = O\left(\frac{\lambda m}{s^{1/2}}\right).
\end{equation}

Let $Z \sim {\rm N}(0, 1)$ be a standard normal random variable.
By the Berry--Esseen Theorem (Theorem~\ref{BE}), we have
\[
 \bigl| \Pr(X'\in I) - \Pr(\mu + \sigma Z \in I) \bigr| \le \frac{2A\rho}{\sigma^3}.
\]

Now by definition $t_1=\mu$, and by assumption $I_1$ and $I_2$ are contained
in an interval of length $m$. Hence, for any $y\in I_1\cup I_2$ we have
\begin{equation}\label{ymu}
 |y-\mu|\le m = O(\sigma).
\end{equation}
It follows that for $j=1,2$ we have
\[
 \Pr(\mu+\sigma Z \in I_j) = 
 \Pr\left(Z \in \left[ \frac{t_j - \mu}{\sigma}, \frac{t_j -\mu + \ell}{\sigma} \right] \right) 
 = \Theta(1) \times \frac{\ell}{\sigma} = \Theta\left(\frac{\ell}{s^{1/2}}\right).
\]
Since $A\rho/\sigma^3 = O(s^{-1/2})$, we thus have
\[
 \Pr(X'\in I_j) =  \Pr(\mu+\sigma Z \in I_j) (1+O(1/\ell)).
\]
The implicit constant here does not depend on $C'$.
Recalling that $\ell\ge C'r=C'$, choosing $C'$ large enough, the $1+O(1/\ell)$ factor here is at least $1/2$,
and it follows by dividing the bound for $j=2$ by that for $j=1$ that
\begin{equation}\label{ratios2}
 \frac{\Pr(X'\in I_2)}{\Pr(X'\in I_1)} = \frac{\Pr(\mu+\sigma Z \in I_2)}{\Pr(\mu+\sigma Z \in I_1)}
  (1+O(1/\ell)).
\end{equation}

Let $\phi(x)=(2\pi)^{-1/2}e^{-x^2/2}$ be the density function
of the standard normal variable. Since $\phi(x)=\phi(0)e^{O(x^2)}$, from \eqref{ymu} we have
\[
 \Pr(\mu+\sigma Z \in I_j) =  \Pr\left(Z \in \left[ \frac{t_j - \mu}{\sigma}, \frac{t_j -\mu + \ell}{\sigma} \right] \right) 
 = \frac{\ell \phi(0)}{\sigma} e^{O(m^2/\sigma^2)}.
\]
Hence,
\begin{equation}\label{ratio3}
 \frac{\Pr(\mu+\sigma Z \in I_2)}{\Pr(\mu+\sigma Z \in I_1)} = e^{O(m^2/\sigma^2)}
 = 1+O(m^2/\sigma^2) = 1+  O\left(\frac{\lambda m}{s^{1/2}}\right),
\end{equation}
recalling \eqref{ms2}. Together, \eqref{ratios}, \eqref{ratios2} and \eqref{ratio3}
complete the proof.
\end{proof}

\section{The core smoothing argument}\label{sec_core}

In this section we prove an ungainly lemma (Lemma~\ref{core}), which is the heart of the smoothing argument.
In this lemma there are many parameters; in the next section we give a simple
choice of parameters that allows us to prove Theorem~\ref{th1}. The reason for keeping
the greater generality here is that it seems (to us) to give a better picture of 
why the method works, and may help in applying the method in other contexts.

So far, it has not mattered whether the intervals we consider are open, closed
or half-open. However, in the arguments below, at certain points we will need
to partition one interval into disjoint intervals of the same type. 
For this reason, from now on we consider only half-open intervals of the form $(a,b]$.

The following definition is key to our smoothing arguments.
Recall that $q_n=\E Q_n$, and that $(Q_n-q_n)/n\dto Q$,
where $Q$ has a continuous positive density function~$f$ on $\RR$.
Given an integer $n\ge 1$
and positive real numbers $m$, $\eps$, and $\Gamma \ge 1$,
we say that the statement $S(n, m, \eps, \Gamma)$ holds\footnote{%
We could work with a statement $S(n, m, L, \eps, \Gamma)$ where, in condition~(i) only,
the interval $I$ is restricted to lie within $[q_n-L,q_n+L]$. In Lemmas~\ref{core}
and~\ref{corebin}, the `input' value of $L$ could be anything (at least $m$),
and the `output' value of $L$ would be the same as the input. Nothing would
change in the proofs.}
if

\medskip
(i) for any half-open interval $I\subset\RR$ of length $m$ and any $x\in \RR$ such that $q_n+nx\in I$
we have $|\Pr(Q_n\in I)-\frac{m}{n}f(x)|\le \eps\frac{m}{n}$, and

\medskip
(ii) for any half-open interval $I$ of length $m$ we have $\Pr(Q_n\in I)\le \Gamma \tfrac{m}{n}$.

\medskip\noindent
Roughly speaking, the fact that $(Q_n-q_n)/n\dto Q$ implies that $S(n, m, \eps, \Gamma)$
holds for some $m=m(n)=o(n)$, some $\eps=\eps(n)\to 0$,
and some constant $\Gamma$. We seek to show that [property~(i) of] $S(n, 1, \eps', \Gamma')$
holds for some slightly larger $\eps' = \eps'(n)$ and $\Gamma'$.

\begin{lemma}\label{core}
There exist positive constants $c$, $C'$, $C$ and $r_0$
such that the following holds.
If the statement $S(n, m, \eps, \Gamma)$ holds, $\ell$ divides $m$,
and there exist real numbers $r$ and $\lambda\ge 1$
such that
\begin{equation}\label{coreconds}
 r_0\le r\le cn, \qquad m\ge \ell\ge C'r \qquad\hbox{and}\quad \lambda m \le \sqrt{r n}, 
\end{equation}
then $S(n, \ell, \eps', \Gamma')$ holds for any $\eps'\ge \eps+ \Gamma \eta$ and 
$\Gamma' \ge \Gamma (1 + \eta)$, where
\begin{equation}\label{etaform}
 \eta = C\left( e^{-c\lambda^2} + \frac{\lambda m}{\sqrt{rn}} + \frac{r}{\ell} + \frac{n}{\ell} e^{-c n/r}\right).
\end{equation}
\end{lemma}
Of course, we could replace $C$ and $C'$ by a single constant $\max\{C,C'\}$, but they
play very different roles in the proof, so we keep them separate.
As we shall see later, the key terms on the right in~\eqref{etaform} are the second and third; we can choose $\lambda = \log n$, say, and then the key conditions to keep $\eta$ small are (roughly stated) that
\begin{equation}\label{key}
 r \ll \ell \le m \ll \sqrt{r n}.
\end{equation}

\begin{proof}
Let $J\subset \RR$ be any interval of length $m$, and let $I_1$ and $I_2$ be subintervals of $J$
with length $\ell$ such that $\Pr(Q_n\in I_1)$ is minimal and $\Pr(Q_n\in I_2)$ is maximal.
We shall show that
\begin{equation}\label{mainaim}
 \Pr(Q_n\in I_2)-\Pr(Q_n\in I_1) \le \eta \Gamma \ell/n.
\end{equation}
Assuming this for the moment, let us show that the lemma follows. To establish property~(i)
of $S(n,\ell,\eps',\Gamma')$, let $I$ be any interval of length $\ell$,
and choose an interval $J$ of length $m$ with $I\subset J$.
Let $x$ be such that $q_n+nx\in I$ and define $I_1$ and $I_2$ as above. By definition,
$\Pr(Q_n\in I_1)\le \Pr(Q_n\in I)\le \Pr(Q_n\in I_2)$. Also,
since $J$ can be partitioned into intervals of length $\ell$,
by simple averaging we have
\begin{equation}\label{IJI}
 \Pr(Q_n\in I_1)\le \frac{\ell}{m} \Pr(Q_n\in J) \le \Pr(Q_n\in I_2).
\end{equation}
By assumption $\Pr(Q_n\in J)$ is within $\eps m/n$ of $f(x)m/n$.
From~\eqref{mainaim} and~\eqref{IJI}
it follows that $\Pr(Q_n\in I_1)$ and $\Pr(Q_n\in I_2)$, and hence $\Pr(Q_n\in I)$,
are within $(\eps \ell / n) + (\eta \Gamma \ell / n) \le \eps' \ell / n$ of $f(x)\ell/n$, as required.

The argument for property (ii) is very similar.  Given any interval $I$ of length $\ell$,
find an interval $J$ of length $m$ containing it, and define $I_1$ and $I_2$ as above.
This time, by assumption, $\Pr(Q_n \in J) \le \Gamma m / n$, so (by averaging)
$\Pr(Q_n \in I_1) \le \Gamma \ell / n$. But then $\Pr(Q_n \in I) \le \Pr(Q_n \in I_2) \le (\Gamma \ell / n) +(\eta \Gamma \ell / n) \le \Gamma' \ell / n$, as required.

It remains to prove \eqref{mainaim}.
Let $r$ and $\lambda\ge 1$ satisfy \eqref{coreconds}.
Increasing~$r$ slightly if necessary (and adjusting the constants in the lemma appropriately),
we may assume that $r$ is an even integer. Let $s=\ceil{c_2n/r}$ where
$c_2$ is as in Lemma~\ref{expand2}.
Write $Q_n=A+B$ where $A$, $B$, and the $\sigma$-algebra $\cF$ are
as in Lemma~\ref{expand2}. The idea is to condition on $\cF$ and use the fact that,
with very high probability, $B$ has conditional distribution in $\cB_{r,s}$ to show that
$\Pr(Q_n\in I_1)$ and $\Pr(Q_n\in I_2)$
are similar. Let $\sigma=\sqrt{rn}$. Since $s=\Theta(n/r)$,
we have $r\sqrt{s}=\Theta(\sigma)$, so by the first part of Lemma~\ref{Brs1},
$\Var[B\mid \cF]=\Theta(r^2 s)=\Theta(\sigma^2)$ (with very high probability).
It will be crucial that $m\ll \sigma$, so that $I_1$ and $I_2$
are not too far apart on the scale over which $B$ varies,
but that $\ell\gg r$.

In the following proof, all statements hold provided $c$ is small enough and $C$ is large enough.
Let $E$ be the event that the conditional distribution of $B$ given $\cF$ is indeed
in $\cB_{r,s}$, so that by Lemma~\ref{expand2}
\begin{equation}\label{Ec}
 \Pr(E)\ge  1-e^{-cn/r}.
\end{equation}
Let $E_1$ be the event that $E$ occurs and the midpoint of $I_2-A$ 
lies in $[-\lambda\sigma,\lambda\sigma]$.
Note that $I_2$ is deterministic, and $A$ is $\cF$-measurable, so $E_1\in \cF$.
Suppose first that $E_1$ occurs. Since $I_1$ and $I_2$ are both subsets of~$J$,
an interval of length $m\le\sigma/\lambda\le \lambda\sigma$,
we have that $I_1-A$ and $I_2-A$ are both contained in $\{x:|x|\le 2\lambda\sigma\}$.
Note that $2\lambda\sigma=O(\lambda r\sqrt{s})$. Also,
$\lambda m/(r\sqrt{s})=O(\lambda m/\sqrt{rn})=O(1)$ by \eqref{coreconds}.
Hence, the conditions of Lemma~\ref{tiltlemma} hold for some constant $K$,
provided we take $C'\ge C'(c_1,K)$. 
When $E_1$ occurs, it thus follows from Lemma~\ref{tiltlemma} that
\begin{equation}\label{B12}
 \Pr(B\in I_2-A\mid \cF) = \Pr(B\in I_1-A\mid \cF) [1+O(\eta_1)],
\end{equation}
where
\[ 
 \eta_1 = (r / \ell) + (\lambda m/ \sigma).
\]

Since $E_1$ is $\cF$-measurable, and $Q_n\in I_i$ if and only if $B\in I_i-A$,
taking the expectation of both sides of \eqref{B12} it follows that
\[
 \Pr\bb{\{Q_n\in I_2\} \cap E_1} = \Pr\bb{\{Q_n\in I_1\} \cap E_1} [1+O(\eta_1)],
\]
so
\begin{equation}\label{c1}
  \Pr\bb{\{Q_n\in I_2\} \cap E_1} - \Pr\bb{\{Q_n\in I_1\} \cap E_1} 
 = O(\eta_1)\Pr(Q_n\in I_1) =O(\eta_1 \Gamma \ell/n),
\end{equation}
since $\Pr(Q_n\in I_1)\le \frac{\ell}{m}\Pr(Q_n\in J)\le \Gamma \frac{\ell}{n}$.

We now consider the `tail case', where $I_2-A$ (and hence $I_1-A$) is far from the mean (zero)
of $B$. We split this case further according to how far.

Assuming purely for convenience that $\lambda$ is an integer, for each integer $y\ge \lambda$
let $E_{2,y}^+$ be the event that $E$ occurs and the midpoint of $I_2-A$ lies in 
$(y\sigma, (y+1)\sigma]$. Similarly, let $E_{2,y}^-$ be the event that $E$ holds
and the midpoint of $I_2-A$ lies in $[-(y+1)\sigma,-y\sigma)$.
Note that
\[
 E = E_1 \cup \bigcup_{y\ge \lambda} E_{2,y}^+ \cup \bigcup_{y\ge \lambda} E_{2,y}^-,
\]
that this union is disjoint, and that all the events involved are $\cF$-measurable.

Fix some $y\ge \lambda$ and suppose that $E_{2,y}^+$ holds. (The argument for $E_{2,y}^-$
will be essentially identical, of course.)
Because $I_2$ has length at most $\sigma$, we see that the left-endpoint of $I_2-A$
is at least $(y-\frac{1}{2})\sigma\ge y\sigma/2$.
Hence, by Lemma~\ref{Brstail},
\begin{equation}\label{I2u}
 \Pr(Q_n\in I_2\mid \cF) = \Pr( B \in I_2-A \mid \cF) \le C \frac{\ell}{\sigma} e^{-c y^2},
\end{equation}
after increasing $C$ and decreasing $c$ if necessary.
Let $J_y=J-y\sigma$, an interval of length $m$ containing $I_2-y\sigma$. Note that (for a given $y$)
the interval $J_y$ is deterministic.
Now, when $E_{2,y}^+$ holds,
$J_y-A$ is an interval of length $m$ contained (recalling $m \leq \sigma$) in $[-2\sigma,2\sigma]$.
By Lemma~\ref{Brs1} it follows that
\[
 \Pr(Q_n\in J_y\mid \cF) = 
\Pr( B\in J_y-A \mid \cF) = \Theta(\mbox{$\frac{m}{\sigma}$}) -O(1/\sqrt{s}).
\]
The implicit constants here depend on $c$, but not on $C'$.
Recalling our assumption $m\ge C'r$, we may thus
choose $C'$ large enough to ensure that the $O(1/\sqrt{s})$ error term is at most half the main term
$\Theta(m/\sigma)=\Theta(m/(r\sqrt{s}))$,
so
\begin{equation}\label{Jl}
 \Pr(Q_n\in J_y\mid \cF) = \Omega(\tfrac{m}{\sigma}).
\end{equation}
Combining \eqref{I2u} and \eqref{Jl}, we see that when $E_{2,y}^+$ holds, then
\begin{equation}\label{BEtail}
 \Pr(Q_n\in I_2\mid \cF) \le \Pr(Q_n\in J_y\mid \cF) O\bb{\tfrac{\ell}{m}e^{-c y^2}}.
\end{equation}
Since $E_{2,y}^+$ is $\cF$ measurable, it follows that
\begin{multline*}
 \Pr\bb{\{Q_n\in I_2\}\cap E_{2,y}^+}
 \leq O\bb{\tfrac{\ell}{m}e^{-c y^2}} \Pr\bb{\{Q_n\in J_y\}\cap E_{2,y}^+} \\
 \leq O\bb{\tfrac{\ell}{m}e^{-c y^2}} \Pr(\{Q_n\in J_y\}) \leq O\bb{\tfrac{\Gamma\ell}{n}e^{-c y^2}},
\end{multline*}
using property (ii) of the statement $S(n, m, \eps, \Gamma)$ for the last step.
A similar bound holds for $E_{2,y}^-$. Since $\sum_{y\ge \lambda} e^{-c y^2} = O(e^{-c\lambda^2})$,
summing we conclude that
\begin{equation}\label{c2}
 \Pr\bb{ \{Q_n\in I_2\} \cap (E\setminus E_1) } = O\bb{\tfrac{\Gamma \ell}{n} e^{-c\lambda^2}}.
\end{equation}

Finally, recalling \eqref{Ec},
\begin{equation}\label{c3}
 \Pr\bb{ \{Q_n\in I_2\} \cap E^\cc } \le \Pr(E^\cc) \le e^{-cn/r}.
\end{equation}
From \eqref{c1}, \eqref{c2}, and \eqref{c3} we conclude that
\[
 \Pr(Q_n\in I_2) \le \Pr(Q_n\in I_1) + \tfrac{\eta \Gamma \ell}{n}
\]
for some $\eta$ that satisfies
\[
 \eta = O\left( \eta_1+e^{-c\lambda^2}+\frac{n}{\Gamma\ell}e^{-cn/r} \right) 
 = O\left( e^{-c\lambda^2} + \frac{\lambda m}{\sqrt{rn}} + \frac{r}{\ell} + \frac{n}{\ell} e^{-c n/r}\right),
\]
recalling for the last step that $\Gamma \ge 1$ by assumption and that $\sigma=\sqrt{rn}$
by definition. This completes the proof of \eqref{mainaim} and thus of the lemma.
\end{proof}

Lemma~\ref{core} will do most of the work for us, but there is a snag. In applying it,
we need to assume $r\ll \ell$. Since we must 
have $r\ge 1$ this means we cannot hope to get down to $\ell=1$ with this method. The fundamental
problem is using the Berry--Esseen Theorem (as in the proof of Lemma~\ref{tiltlemma}) to try to get a good bound on the 
probability that an integer-valued random variable is in an interval of length
$1$---since the assumptions of 
the theorem do not distinguish between intervals such as $[k,k+1]$ (which contains two integers)
and $[k-1/2,k+1/2]$ (which contains one),
we can't hope for a good approximation in this way. The solution in this case was outlined near the start
of the paper: for this part of the argument we work not with a binomial-like distribution,
but with a binomial distribution. Then we can calculate the relevant probabilities
directly, avoiding the approximation in the Berry--Esseen Theorem.
This is captured in Lemma~\ref{corebin} below, whose proof is a variant of the proof
of Lemma~\ref{core}. Before coming to this lemma, we give the decomposition result
that we shall need.

\begin{lemma}\label{expandbin}
There are constants $c>0$ and $n_0$ such that for any $n\ge n_0$ we may write
$Q_n=A+B$ where, for some $\sigma$-algebra $\cF$, we have that $A$
is $\cF$-measurable, and with probability at least $1-e^{-cn}$,
the conditional distribution of $B$ given $\cF$ is the binomial distribution $\Bi(\ceil{cn},2/3)$.
\end{lemma}
Of course, this lemma can be rephrased to say that there are independent random variables
$A$ and $B$, with $B\sim \Bi(\ceil{cn},2/3)$, such that with probability $1-e^{-cn}$ we have $Q_n=A+B$.
We keep the wording above to strengthen the analogy to Lemmas~\ref{expandA} and~\ref{expand2}.
\begin{proof}
For $n=3$, {\tt QuickSort} either needs two comparisons (if the initial pivot happens to be the middle
element) or, with probability $2/3$, three comparisons.
A simple variant of the proof of Lemma~\ref{expandA} shows that if $c>0$ is small enough,
then with probability at least $1-e^{-cn}$ we may partially expand the execution
tree of {\tt QuickSort} run on a list of $n$ elements so as to leave $\ceil{cn}$ instances of {\tt QuickSort}
of size $3$. We take $B$ to be the number of comparisons in these instances minus $2\ceil{cn}$.
\end{proof}

\begin{lemma}\label{corebin}
There exist constants $n_0$, $c$ and $C$ such that the following holds.
Let $n\ge n_0$ be an integer and let $m\ge 1$ and $\lambda\ge 1$ be real numbers such that
\begin{equation}\label{condbin}
 \lambda m\le \sqrt{n} \quad\hbox{ and }\quad \lambda\le c\sqrt{n}/20.
\end{equation}
If $S(n, m, \eps, \Gamma)$ holds, then so does $S(n, 1, \eps', \Gamma')$
for any $\eps' \ge \eps + \Gamma \eta$ and $\Gamma' \ge \Gamma (1 + \eta)$,
where
\begin{equation}\label{etabin}
 \eta = C\left( e^{-c\lambda^2} + \frac{\lambda m}{\sqrt{n}} + n e^{-cn} \right).
\end{equation}
\end{lemma}
\begin{proof}
We shall show that whenever $x_1$ and $x_2$ are integers with $|x_2-x_1|\le m$, then
\begin{equation}\label{mainaim2}
 \Pr(Q_n=x_2) - \Pr(Q_n=x_1) \le \eta \Gamma / n.
\end{equation}
The result then follows roughly as in the proof of Lemma~\ref{core}; since there
is a small twist to deal with non-integer $m$, we briefly outline the argument.

First, to establish property (ii) of $S(n,1,\eps',\Gamma')$, let $x$ be any integer,
and pick an interval $J\subset\RR$ of length $m$ containing $x$, with
$\ceil{m}$ integers in $J$. By averaging, there is some $x_1\in J$ such that
\[
 \Pr(Q_n=x_1)\le \Pr(Q_n\in J)/\ceil{m}\le \Pr(Q_n\in J)/m \le \Gamma/n,
\]
using the assumption $S(n,m,\eps,\Gamma)$ in the last step.
Applying \eqref{mainaim2} with $x_2=x$ gives $\Pr(Q_n=x)\le \Gamma'/n$, as required.
For property (i), using the same $J$ and the bound $\Pr(Q_n\in J)\le (f(y)+\eps)m/n$
where $y=(x-q_n)/n$ gives $\Pr(Q_n=x)\le (f(y)+\eps')/n$.
For the lower bound, consider an interval $J'$ of length $m$ containing $x$,
with $\floor{m}$ integers in $J$.
Then find $x_2\in J$ with $\Pr(Q_n=x_2)\ge \Pr(Q_n'\in J)/m$ and apply
\eqref{mainaim2} with $x_1=x$.

It remains to prove \eqref{mainaim2}.
We follow the proof of Lemma~\ref{core}, but replacing the distribution of class $\cB_{r,s}$
by a binomial distribution $\Bi(s,p)$ where $s=\Theta(n)$ and $p=2/3$. (Any $p$ bounded
away from $0$ and $1$ would work.)
The existence of the relevant decomposition is given by Lemma~\ref{expandbin};
let $s=\ceil{cn}$ where $c$ is as in that lemma, and
let $E$ be the event that the conditional distribution of $B$ is indeed $\Bi(s,2/3)$,
so $\Pr(E)\ge 1-e^{-cn}$.

Suppose that~$E$ occurs.  Then for $0 \le k \le s - 1$ we have
\begin{equation}\label{binrat1}
 \frac{\Pr(B=k+1 \mid \cF)}{\Pr(B=k \mid \cF)} = \frac{\tbinom{s}{k+1}}{\tbinom{s}{k}} \frac{p}{1-p}
 = \frac{s-k}{k+1}\frac{p}{1-p}
 = \frac{1-(k/s)}{1-p} \frac{p}{(k/s)+(1/s)}.
\end{equation}
If $s/100\le k\le 99s/100$, say, then it follows that
\[
 \frac{\Pr(B=k+1 \mid \cF)}{\Pr(B=k \mid \cF)}  = 1+O\bb{|(k / s) - p| + (1 / s)}.
\]
Recalling that $s=\Theta(n)$, when $k$ is within $2\lambda\sqrt{n} \le c n / 10 \le s / 10$ of $p s$ this gives $\Pr(B=k+1 \mid \cF)/\Pr(B=k \mid \cF) = 1+O(\lambda/\sqrt{n})$ (uniformly in such~$k$).
It follows that if $k_1$ and $k_2$ satisfy $|k_i-ps|\le 2\lambda\sqrt{n}$ and
$|k_2-k_1|\le m$, then
\begin{equation}\label{binsim}
 \Pr(B=k_2 \mid \cF)=\Pr(B=k_1 \mid \cF)\bb{1+O(\lambda m/\sqrt{n})}.
\end{equation}
Let $E_1$ be the event that $E$ occurs and $x_2-A$ is within $\lambda\sqrt{n}$ of the mean $ps$ of $B$.
Since $m\le \sqrt{n}/\lambda \le\lambda\sqrt{n}$, this implies that $|(x_1-A)-ps|\le 2\lambda\sqrt{n}$,
say. Using \eqref{binsim} we see that when $E_1$ holds,
then
\[
 \Pr(Q_n=x_2\mid \cF)= \Pr(Q_n=x_1\mid \cF) \bb{1+O(\lambda m/\sqrt{n})}
\]
and it follows [arguing as for \eqref{c1}] that
\begin{equation}\label{binc1}
 \Pr\bb{\{Q_n=x_2\}\cap E_1} -  \Pr\bb{\{Q_n=x_1\}\cap E_1} = O\bb{(\lambda m/\sqrt{n}) \Gamma/n}.
\end{equation}

As in the proof of Lemma~\ref{core}, for integer $y\ge\lambda$ let $E_{2,y}^+$ be the event that
$E$ occurs and $x_2-A$ lies in $(ps+y\sigma, ps+(y+1)\sigma]$, where now $\sigma=\sqrt{n}$.
By an elementary calculation [using \eqref{binrat1} as a starting point, for example],
letting $k=x_2-A$ and recalling that $s=\Theta(n)$, there exist positive constants $c'$ and $c$
such that whenever $E_{2,y}^+$ holds then
\begin{multline*}
 \Pr(B=x_2-A\mid \cF) = \Pr(\Bi(s,p)=k) \\
 \le \Pr(\Bi(s,p)=\floor{ps}) e^{-c' y^2\sigma^2/s} \le C n^{-1/2} e^{-cy^2}.
\end{multline*}
As before (compare the definition of $J_y$ in the proof of Lemma~\ref{core}), let $J_y$ be an interval of length $m$ containing $x_2-y\sigma$. Then, when $E_{2,y}^+$ holds,  $J_y-A$
is an interval of length $m$ contained in $[ps-2\sigma,ps+2\sigma]$, say, and it follows
using elementary properties of the binomial distribution that
\[
 \Pr(Q_n\in J_y\mid \cF) = \Pr(\Bi(s,p)\in J_y-A\mid \cF) = \Theta(m/\sigma) = \Theta(m/\sqrt{n}).
\]
It follows that when $E_{2,y}^+$ holds, then
\begin{equation}\label{bintail}
 \Pr(Q_n=x_2 \mid \cF) \le \Pr(Q_n\in J_y\mid \cF) O\bb{\tfrac{1}{m}e^{-c y^2}},
\end{equation}
the analogue of \eqref{BEtail}.
The rest of the proof follows exactly that of Lemma~\ref{core}; we omit the details,
noting only that the error terms arise as follows: $e^{-c\lambda^2}$ from \eqref{bintail}
[just as from \eqref{BEtail}], $\lambda m/\sqrt{n}$ from \eqref{binc1}, and $ne^{-cn}$
from the probability that $E$ fails (recalling that our error terms are written relative to
$\Gamma / n$).
\end{proof}

Note that in Lemma~\ref{corebin}, there is no error term corresponding to the $r / \ell$
term in Lemma~\ref{core}, which can be traced back to the approximation error from applying
the Berry--Esseen Theorem in Lemma~\ref{tiltlemma}.

\section{Completing the proof}\label{sec_complete}

In this section we prove Theorem~\ref{th1} using Lemmas~\ref{core} and~\ref{corebin},
together with the following result of Fill and Janson~\cite{FJ2, FJ1}.

\begin{theorem}\label{thstart}
Let $F_n$ denote the distribution function of $(Q_n-q_n)/n$
and $f$ the continuous density function of the limiting distribution $Q$.
There is a constant $C$ such that for any $x$ and any $n\ge 1$ we have
\begin{equation}\label{FJstart}
 \left| \frac{F_n(x+\tfrac{\delta_n}{2})-F_n(x-\tfrac{\delta_n}{2})}{\delta_n} - f(x) \right| 
\le Cn^{-1/6},
\end{equation}
where $\delta_n=2C n^{-1/6}$. Furthermore, $f$ is differentiable on $\RR$, and we have
\begin{equation}\label{start1}
 |f(x)| \le 16 \quad\hbox{and}\quad  |f'(x)| \le \widetilde{C} := 2466
\end{equation}
for all $x\in \RR$.
\end{theorem}
The first (main) statement is part of \cite[Theorem 6.1]{FJ1}; the
second statement is from \cite[Theorem~3.3]{FJ2}. Of course, the particular
values of the constants will not be relevant here.

Rephrased, \eqref{FJstart} says that for any (half-open, as usual) interval $I$ of length
$m=\delta_n n = 2C n^{5/6}$, we have
\begin{equation}\label{prestart}
 | \Pr(Q_n\in I) - \tfrac{m}{n}f(x_I) | \le Cn^{-1/6} \tfrac{m}{n},
\end{equation}
where $x_I$ is such that $q_n+nx_I$ is the midpoint of $I$.
This is almost, but not quite, condition (i) of the statement $S(n,m,\eps,\Gamma)$
defined before Lemma~\ref{core}.

\begin{corollary}\label{Cstart}
Let $C$ and $\widetilde{C}$ be as in Theorem~\ref{thstart}, and set $C'=C+\widetilde{C}C$.
If $n$ is large enough then $S(n,m,\eps,\Gamma)$ holds with
$m=2Cn^{5/6}$, $\eps=C'n^{-1/6}$, and $\Gamma=17$.
\end{corollary}
\begin{proof}
Let $I$ be any (half-open, as always) interval of length $m$. To establish property (i) of
$S(n,m,\eps,\Gamma)$ we must show that
\begin{equation}\label{start}
  | \Pr(Q_n\in I) - \tfrac{m}{n}f(x) | \le C'n^{-1/6} \tfrac{m}{n}
\end{equation}
for any $x$ such that $q_n+nx\in I$. Let $x_I$ be such that $q_n+nx_I$ is the midpoint
of $I$. Then
\[
 |x-x_I| = \frac{|(q_n+nx)-(q_n+nx_I)|}{n} \le \frac{m}{2n} = Cn^{-1/6}.
\]
By the Mean Value Theorem and the second bound in \eqref{start1} we have
\[
 |f(x) - f(x_I)| \le \widetilde{C} |x - x_I| \leq \widetilde{C}Cn^{-1/6}.
\]
This, \eqref{prestart} and the triangle inequality imply \eqref{start}.

To establish statement~(ii) of $S(n, m, \eps, \Gamma)$ simply note that by \eqref{start}
and \eqref{start1} we have
\[
 \Pr(Q_n\in I) \le \tfrac{m}{n}f(x) + C'n^{-1/6} \tfrac{m}{n} \le (16+C'n^{-1/6}) \tfrac{m}{n}
 \le 17 \tfrac{m}{n},
\]
if $n$ is large enough.
\end{proof}

We are now ready to prove Theorem~\ref{th1}.
\begin{proof}[Proof of Theorem~\ref{th1}]
We will show that the theorem holds with the error term $O(n^{-1 - \eps})$ replaced by 
$O(n^{-1 - (1/18)} \log n)$. To do this, it suffices to establish that $S(n,1,\eps',\Gamma')$
holds for some $\eps'=O(n^{-1/18} \log n)$; condition (i) of this statement
is exactly \eqref{th1eq}.
We establish this by using
\[
 K := \floor{\tfrac{1}{2}\log_2 n} + 1
\]
rounds of smoothing, as we now explain in some detail.

In round~$k$, for $1 \leq k \leq K - 1$, we will apply Lemma~\ref{core} with parameters
\[
 (m, \eps, \Gamma, \ell, r, \lambda) = (m_k, \eps_k, \Gamma_k, \ell_k, r_k, \log n),
\]
to be specified in a moment.
Let the constants $C$ and $C'$ be as in Corollary~\ref{Cstart}.
We set
\[
 m_k := 4 C n^{5 / 6} 2^{-k}, \quad 1 \leq k \leq K,
\]
\[
 \ell_k := m_{k+1} = 2 C n^{5 / 6} 2^{-k}, \quad 1 \leq k \leq K - 1,
\]
and
\[
 r_k := \frac{(m_k \ell_k)^{2 / 3}}{n^{1 / 3}} = \Theta(n^{7/9}2^{-4k/3}) , \quad 1 \leq k \leq K - 1.
\]
Furthermore, we set
\[
 \eps_1 := C' n^{-1/6} \quad\hbox{and} \quad \Gamma_1 := 17,
\]
and
\[
 \eta_k := \widehat{C} 2^{-k/3} n^{-1/18} \log n, \quad 1 \leq k \leq K - 1,
\]
for a constant $\widehat{C}$ to be chosen in a moment. Then we inductively define
\[
 \eps_k := \eps_{k - 1} + \Gamma_{k - 1} \eta_{k - 1}, \quad 2 \leq k \leq K
\]
and
\[
 \Gamma_k := \Gamma_{k - 1} (1 + \eta_{k - 1}), \quad 2 \leq k \leq K.
\] 

The divisibility condition $\ell_k|m_k$ and the conditions \eqref{coreconds}
of Lemma~\ref{core} are easily seen to hold for $n$ large enough. Moreover,
we have
\[
 \lambda \frac{m_k}{\sqrt{r_kn}} = \Theta\bb{2^{-k/3}n^{-1/18}\log n} 
\quad\hbox{and}\quad \frac{r_k}{\ell_k} = \Theta\bb{2^{-k/3}n^{-1/18}}.
\]
(Indeed, given $m_k$ and $\ell_k$, we have chosen $r_k$ to balance these terms,
ignoring the slowly varying factor $\lambda = \log n$.)
Since the outer two terms in \eqref{etaform} are superpolynomially small, we
see that if $\widehat{C}$ is chosen suitably large,
then in each application of Lemma~\ref{core} we have $\eta\le \eta_k$.
Since the statement $S(n,m_1,\eps_1,\Gamma_1)$ holds by Corollary~\ref{Cstart},
applying Lemma~\ref{core} $K-1$ times we conclude that
$S(n,m_K,\eps_K,\Gamma_K)$ holds.

Note that $m_K=\Theta(n^{5/6} 2^{-K})=\Theta(n^{1/3})$.
In the final round we apply Lemma~\ref{corebin} with
\[
 (m, \eps, \Gamma, \lambda) = (m_K, \eps_K, \Gamma_K, \log n).
\]
The conditions \eqref{condbin} hold with room to spare (for $n$ large enough). The
quantity $\eta$ appearing in \eqref{etabin} is $O(n^{-1/6}\log n)=o(n^{-1/18})$,
so we conclude that
$S(n, 1, \eps', \Gamma')$ holds, with $\Gamma' = O(\Gamma_K)$
and $\eps' = \eps_K+o(n^{-1/18} \Gamma_K)$.

It remains to estimate $\Gamma_K$ and $\eps_K$.  First,
\[
 \Gamma_K = \Gamma_1 \prod_{j = 1}^{K - 1} (1+\eta_j)
 \leq \Gamma_1 \exp\left( \sum_{j = 1}^{K - 1} \eta_j \right)
 \le \Gamma_1 e^{5 \eta_1} \sim \Gamma_1 \text{\ as $n \to \infty$},
 \]
so (for $n$ large enough), $\Gamma_k\le 18$ for all $k$.
Then
\[
 \eps_K \leq \eps_1 + 18 \sum_{j = 1}^{k - 1} \eta_j = \eps_1+ O(\eta_1) = O(n^{-1/18} \log n).
\]
It follows that $\eps'=O(n^{-1/18}\log n)$, completing the proof.
\end{proof}

\section{A softer version}\label{sec_soft}

We describe here an argument for a weaker version of Theorem~\ref{th1}. The advantage of this argument is that it requires less as input:\ only a distributional limit theorem, not one with the explicit
rate of convergence in Theorem~\ref{thstart}. This may be useful in other contexts.

\begin{theorem}\label{th1weak}
Uniformly in integers~$x$
we have
\[
 \Pr(Q_n = x) =n^{-1} f((x - q_n) / n) + o(n^{-1})
\]
as $n\to\infty$, where $f$ is the continuous probability density function of $Q$.
\end{theorem}

\begin{proof}
We take as our starting point that $(Q_n-q_n)/n\dto Q$, making no assumption about
the rate of convergence. We do assume certain properties, immediate from \cite[Theorem~3.3]{FJ2}, of the density function~$f$ 
of~$Q$,
namely that $f$ is bounded (by $M$, say) and uniformly continuous on $\RR$.
The only other properties of $Q_n$ we use are the decompositions provided by Lemmas~\ref{expand2}
and~\ref{expandbin}. This is all we need to prove Lemmas~\ref{core} and~\ref{corebin} exactly
as above. The difference to the argument in Section~\ref{sec_complete} is how we apply
these lemmas.

Let $F_n$ be the distribution function of the normalized distribution $Q_n^*=(Q_n-q_n)/n$,
and let $F$ be that of $Q$. Since $Q_n^*\dto Q$ and $F$ is continuous, we have 
$F_n\to F$ in sup-norm by Poly\`{a}'s theorem (for example, \cite[Exercise~4.3.4]{Chung}). In other words, there is some $\delta(n)\to 0$ such that
for all $x$ and $n$ we have 
\begin{equation}\label{Fconv}
 |F_n(x)-F(x)|\le \delta(n).
\end{equation}
Let
\begin{equation}\label{fdiff}
 \gamma(n) := \sup_{x,y\::\:|x-y|\le \delta(n)^{1/2}} |f(x)-f(y)|.
\end{equation}
Since $\delta(n)^{1/2}\to 0$ and $f$ is uniformly continuous, $\gamma(n)\to 0$ as $n\to\infty$.
For any interval $I$ of length $|I| = \delta(n)^{1/2}$, by \eqref{Fconv} we have that $\Pr(Q_n^*\in I)$
is within $2\delta(n)$ of $\int_I f(x)\mathrm{d}x$ which, by \eqref{fdiff}, is within
$\gamma(n)|I|$ of $f(x)|I|$ for any $x\in I$.
Thus, $\Pr(Q_n^*\in I)$ is within $[2\delta(n)^{1/2}+\gamma(n)]|I|$ of $f(x)|I|$.
Replacing $Q_n^*$ by $Q_n$ and $I$ by $q_n+nI$, this says exactly that property~(i)
of $S(n, m_0, \eps_0, \Gamma_0)$ holds, where
$m_0=m_0(n)=n \delta(n)^{1/2}$ and $\eps_0=\eps_0(n)=2\delta(n)^{1/2}+\gamma(n)$.
Taking $\Gamma_0 = M + 1$, which is at least $M+\eps_0(n)$ for $n$ large enough, we also have property (ii).

To summarize, the distributional limit theorem (plus assumptions on $f$) gives us that there
exist $m_0=o(n)$, $\eps_0=o(1)$ and $\Gamma_0 = O(1)$ such that $S(n, m_0, \eps_0, \Gamma_0)$ holds. We now aim
to apply Lemma~\ref{core} as many times as necessary.
The key point is that if $S(n,m,\eps,\Gamma)$ holds where $m=n/\omega$, with $\omega\to\infty$,
then in one step we can roughly square $\omega$. More precisely, set $\ell=n/\omega^{1.5}$, say.
Then to satisfy \eqref{key} we can take say $r=n/\omega^{1.8}$, so
\[
 \frac{m}{\sqrt{rn}} = \omega^{-0.1}\quad\hbox{and}\quad \frac{r}{\ell} = \omega^{-0.3}.
\]
Choosing $\lambda=\log\omega$, say,
the $e^{-c\lambda^2}$ term in \eqref{etaform} is superpolynomially small in~$\omega$. 
The term $(n/\ell)e^{-cn/r}$
is $\omega^{1.5} e^{-c\omega^{1.8}}$, which tends to zero extremely quickly as $\omega$ grows.
The conclusion is that the conditions of Lemma~\ref{core} will hold, with $\eta=O(\omega^{-0.1}\log\omega)=O(\omega^{-0.05})$, say.
Applying the lemma repeatedly, in the $i$th application we have
$\omega=\omega_i=\omega_0^{1.5^i}$;
we stop when $\ell$ is no more than $n^{0.4}$, say 
(and hence, provided $\omega_0 \leq n^{0.9}$, which we may presume without loss of generality, is at least $n^{0.1}$).
Since $\sum_i \omega_i^{-0.05}=o(1)$ (recalling that $\omega_0\to\infty$), the sum of the error terms
$\eta$ is $o(1)$, and we find that some $S(n, m, \eps, \Gamma)$ holds with $n^{0.1}\le m\le n^{0.4}$,
$\eps=o(1)$ and $\Gamma = O(1)$. A single application of Lemma~\ref{corebin}, say with 
$\lambda=\log n$, yields $S(n, 1, \eps', \Gamma')$ where also $\eps'=o(1)$, completing the proof.
\end{proof}

\end{document}